\documentclass{aims}
\usepackage[mathscr]{euscript}
\usepackage{verbatim}
\usepackage{color}
\usepackage[all]{xy}
\usepackage{amsmath}
\usepackage{paralist}
\usepackage{graphics} 
\usepackage{epsfig} 
\usepackage{graphicx}
\usepackage{epstopdf}
\usepackage[colorlinks=true]{hyperref}
\hypersetup{urlcolor=blue, citecolor=red}

\def\currentvolume{X}
 \def\currentissue{X}
  \def\currentyear{200X}
   \def\currentmonth{XX}
    \def\ppages{X--XX}
     \def\DOI{10.3934/xx.xx.xx.xx}

\theoremstyle{plain}
\newtheorem{thm}{Theorem}[section]
\newtheorem{lem}[thm]{Lemma}
\newtheorem{prop}[thm]{Proposition}
\newtheorem{cor}[thm]{Corollary}
\newtheorem{claim}{Claim}
\theoremstyle{definition}
\newtheorem{rem}[thm]{Remark}
\newtheorem{defn}[thm]{Definition}

\newcommand{\N}{\mathbb{N}}
\newcommand{\F}{\mathscr{F}}
\newcommand{\B}{\mathscr{B}}

\newcommand{\PP}{\mathscr{P}}
\newcommand{\A}{\mathscr{A}}

\DeclareMathOperator{\supp}{supp}
\DeclareMathOperator{\Orb}{Orb}
\DeclareMathOperator{\Trans}{Trans}
\DeclareMathOperator{\Perf}{Perf}
\DeclareMathOperator{\diam}{diam}
\newcommand{\pubd}{\mathrm{pubd}}
\newcommand{\ps}{\mathrm{ps}}
\newcommand{\ip}{\mathrm{ip}}
\newcommand{\htop}{h_{\mathrm{top}}}

\begin{document}
\title{Localization of mixing property via Furstenberg families}

\author[Jian Li]{}
\email{lijian09@mail.ustc.edu.cn}

\thanks{The author was supported in part by Scientific Research Funds of Shantou University (YR13001),
Guangdong Natural Science Foundation (S2013040014084) and NNSF of China (11401362, 11471125).}

\maketitle

\centerline{\scshape Jian Li}
\medskip
{\footnotesize
 \centerline{Department of Mathematics, Shantou University, Shantou,}
   \centerline{Guangdong 515063, P.R. China}
} 

\bigskip


\begin{abstract}
This paper is devoted to studying the localization of mixing property via Furstenberg families.
It is shown that there exists some
$\F_{\pubd}$-mixing set
in every dynamical system with positive entropy,
and some $\F_{\ps}$-mixing set in every non-PI minimal system.
\end{abstract}

\section{Introduction}
Throughout this paper a \emph{topological dynamical system}
(or \emph{dynamical system}, \emph{system} for short) is a pair $(X, T)$,
where $X$ is a non-empty compact metric space with a metric $d$ and
$T$ is a homeomorphism from $X$ to itself.

Let $(X,T)$ be a dynamical system.
For two subsets $U$ and $V$ of $X$, we define the \emph{hitting time set} of $U$ and $V$ by
\[N(U,V)=\{n\in\N\colon U\cap T^{-n}(V)\neq\emptyset\}.\]
We say that $(X,T)$ is \emph{topologically transitive} (or just \emph{transitive})
if for every two non-empty open subsets $U$ and $V$ of $X$, the set $N(U,V)$ is not empty;
\emph{weakly mixing} if the product system $(X\times X,T\times T)$ is transitive;
\emph{strongly mixing} if for every two non-empty open subsets $U$ and $V$ of $X$, the set $N(U,V)$ is cofinite,
that is there exists $N\in\N$ such that $\{N,N+1,N+2,\dotsc\}\subset N(U,V)$.

It was shown by Furstenberg that if $(X,T)$ is weakly mixing,
then it is weakly mixing of all finite orders, that is for any $n\geq 2$ and any non-empty open subsets $U_1,U_2,\dotsc,U_n$ and $V_1,V_2,\dotsc,V_n$ of $X$,
\[\bigcap_{i=1}^n N(U_i,V_i)\neq\emptyset.\]

In~\cite{XY92}, Xiong and Yang characterized mixing properties by Xiong chaotic sets.
More specifically, a subset $K$ of $X$ is called a \emph{Xiong chaotic set} with respect to
an increasing sequence $\{p_i\}$ of positive integers if for any
subset $E$ of $K$ and for any continuous map $g\colon E \to X$ there is a subsequence $\{q_i\}$ of $\{p_i\}$
such that $\lim_{i\to\infty}T^{q_i}(x)=g(x)$ for every $x\in E$.
They showed that a non-trivial dynamical system $(X,T)$
 is weakly mixing if and only if
there exists some $c$-dense $F_\sigma$-subset $K$ of $X$
that is Xiong chaotic with respect to the sequence $1,2,3,\dotsc$;
 strongly mixing if and only if for any increasing sequence of
positive integers there is a $c$-dense $F_\sigma$-subset $K$ of $X$ that is Xiong chaotic with respect to
this sequence,
where by $c$-dense we mean $K\cap U$ is an uncountable set for any non-empty open subset $U$ of $X$.

Inspired by Xiong chaotic sets, Blanchard and Huang introduced the concept of weakly mixing sets,
which can be regarded as the localization of weak mixing~\cite{BH08}.
A closed subset $A\subset X$ is called a \emph{weakly mixing set}
if there exists a dense Mycielski subset (union of countable many Cantor sets)
$B$ of $A$ such that for any subset $E$ of $B$ and
any continuous map $g\colon E\to A$ there exists a sequence $\{q_i\}$ of positive integers such that
$\lim_{i\to\infty} T^{q_i}(x)=g(x)$ for every $x\in E$.
An alternative definition of weakly mixing sets is using hitting time sets.
Let $A$ be a closed subset of $X$ with at least two point.
Then $A$ is a weakly mixing subset of $X$ if and only if for any
$n\geq 1$ and any open subsets $U_1,U_2,\dotsc,U_n$ and $V_1,V_2,\dotsc,V_n$ of $X$ intersecting $A$,
\[\bigcap_{i=1}^n N(U_i\cap A,V_i)\neq\emptyset.\]
It is shown in~\cite{BH08} that a topological dynamical system with positive entropy contains many weakly mixing sets.
Recently, there is a series of work on the study of weak mixing sets of finite order and relative dynamical properties,
see \cite{LOZ13,OZ11,OZ12,OZ13}.

As the hitting time set $N(U,V)$ is a subset of positive integers,
we can use some class of sets of positive integers to classify transitive systems.
A collection $\F$ of subsets of positive integers is called a \emph{Furstenberg family} (or just \emph{family}),
if it is hereditary upward, i.e., $F_1\subset F_2$ and $F_1\in\F$ imply $F_2\in \F$.
A dynamical system $(X,T)$ is called \emph{$\F$-transitive} if for any two non-empty open subsets
$U$ and $V$ of $X$, $N(U,V)\in\F$.
We say that $(X,T)$ is \emph{$\F$-mixing} if $(X\times X,T\times T)$ is $\F$-transitive.
In~\cite{HSY04}, Huang, Shao and Ye extended Xiong-Yang's result
to $\F$-mixing systems,
that is a non-trivial dynamical system $(X,T)$ is $\F$-mixing
if and only if for every $S\in\kappa\F$ (the dual family of $\F$)
there exists a dense Mycielski subset $K$ of $X$ which is Xiong chaotic with respect to $S$.

In this paper, we are devoted to studying the localization of mixing property via Furstenberg families.
The paper is organized as follows.
In Section 2, we recall some notions and aspects of topological dynamics.
In Section 3, we introduce the concept of $\F$-mixing sets and characterize them by Xiong chaotic sets.
In Section 4, we study dynamical systems with positive  entropy.
We show that if an ergodic invariant measure $\mu$ has positive entropy then
the collection of $\F_{\pubd}$-mixing sets
is residual in the collection of $\mu$-entropy sets,
where $\F_{\pubd}$ is the Furstenberg family of sets with positive upper Banach density.
As a consequence, if a dynamical system has positive entropy then
the collection of $\F_{\pubd}$-mixing sets
is dense in the collection of entropy sets.
In Section 5, we study factor maps between dynamical systems.
We show that if an intrinsic factor map is $\F$-point mixing,
then the collection of $\F$-mixing sets is residual in the fiber space.
We also apply this result to show that a minimal non-PI system contains some
$\F_{\ps}$-mixing set,
where $\F_{\ps}$ is the Furstenberg family of piecewise syndetic sets.
\section{Preliminary}
In this section, we provide some basic notations, definitions
and results which will be used later.
\subsection{Furstenberg family}
For the set of positive integers $\N$,
denote by $\PP=\PP(\N)$ the collection of all subsets of $\N$.
A subset $\F$ of $\PP$ is called a \emph{Furstenberg family} (or just \emph{family}),
if it is hereditary upward, i.e., $F_1\subset F_2$ and $F_1\in\F$ imply $F_2\in \F$.
A family $\F$ is \emph{proper} if it is a non-empty proper subset of $\PP$,
i.e., neither empty nor all of $\PP$.
It is easy to see that a family $\F$ is proper
if and only if $\N\in \F$
and $\emptyset \not\in \F$.
Any non-empty collection $\A$ of
subsets of $\F$ generates a family
$\F(\A)=\{F\subset\N\colon F\supset A $
for some $ A \in\A\}$.

For a family $\F$, the \emph{dual family} is
$\kappa\F=\{F\subset \N\colon F\cap F'\neq \emptyset, \forall F'\in \F\}$.
Clearly, $\kappa\F$ is a family or a proper family if $\F$ is.
It is also not hard to see that $\kappa\kappa\F=\F$.
Let $\F_{\mathrm{inf}}$ be the family of all infinite subsets of $\N$.
It is easy to see that its dual family $\kappa\F_{\mathrm{inf}}$
is the family of all cofinite subsets, denoted by $\F_{\mathrm{cf}}$.

\textbf{Convention}: All the families considered in this paper are assumed to be proper and contained in $\F_{\mathrm{inf}}$.

We say that a subset $F$ of $\N$ is an \emph{IP-set} if there is a subsequence $\{p_i\}$ of $\N$ such that
$\{p_{i_1}+\dotsb+p_{i_k}\colon i_1<\dotsb<i_k, n\in\N\}\subset F$.
The family of IP-sets is denoted by $\F_{\ip}$.

Let $F$ be a subset of $\N$. The \emph{upper Banach density of $F$} is defined by
\[BD^*(F)=\limsup_{|I|\to\infty}\frac{|F\cap I|}{|I|},\]
where $I$ is taken over all non-empty finite intervals of $\N$.
The family of sets with positive upper Banach density is denoted by
$\F_{\pubd}=\{F\subset \N\colon BD^*(F)>0 \}$.

A subset of $F$ of $\N$ is \emph{thick} if it contains arbitrarily long runs of positive integers,
i.e., for every $n\in\N$ there exists some
$k_n\in\N$ such that
$\{k_n, k_n+1,\ldots, k_n+n\}\subset F$;
\emph{syndetic} if there is $N\in\N$
such that $\{i,i+1,\ldots, i+N\}\cap F\neq \emptyset$ for any $i\in\N$;
\emph{piecewise syndetic} if it is the intersection of a thick set and a syndetic set.
The family of syndetic sets, thick sets and piecewise syndetic sets
are denoted by $\F_{\mathrm{s}}$, $\F_{\mathrm{t}}$ and $\F_{\ps}$, respectively.

\subsection{Topological dynamics}
Let $(X,d)$ be a compact metric space. For $x\in X$ and $\varepsilon>0$,
denote $B(x,\varepsilon)=\{y\in X\colon d(x,y)<\varepsilon\}$.
For $n\geq 2$, denote the $n$-th product space $X^{n}=X\times X\times \dotsb \times X$ ($n$-times)
and the diagonal $\Delta_n(X)=\{(x,x,\dotsc,x)\in X^n\colon x\in X\}$.
Let $M(X)$ be the set of regular Borel probability measures on $X$.
The \emph{support} of a measure $\mu\in M(X)$, denoted by $\supp(\mu)$,
is the smallest closed subset $C$ of $X$ such that $\mu(C)=1$.

Let $(X,T)$ be a topological dynamical system.
If $Y$ is a non-empty closed subset of $X$ and $TY\subset Y$, then $(Y,T)$ forms a subsystem of $(X,T)$.
For $n\geq 2$, denote the $n$-th product system by $(X^n,T^{(n)})=(X\times X\times \dotsb\times X,
T\times T\times\dotsb\times T)$ ($n$-times).

The \emph{orbit} of a point $x\in X$, $\{x,T(x),T^2(x),\dotsc\}$, is denoted by $\Orb(x,T)$.
We say that  a point $x\in X$ is a \emph{periodic point} of $(X,T)$ if $T^n(x)=x$ for some $n\in\N$;
a \emph{recurrent point} of $(X,T)$ if there exists an increasing sequence $\{k_n\}$ of positive integers
such that $T^{k_n}(x)\to x$ as $n\to\infty$;
a \emph{transitive point} of $(X,T)$ if $\Orb(x,T)$ is dense in $X$.
Denote by $\Trans(X,T)$ the set of all transitive points of $(X,T)$.
It is well known that if $(X,T)$ is transitive, then the $\Trans(X,T)$ is a dense $G_\delta$ subset of $X$.
A dynamical system $(X,T)$ is \emph{minimal} if $\Trans(X,T)=X$. A point $x\in X$ is \emph{minimal} if
$(\overline{Orb(x,T)},T)$ is a minimal subsystem of $(X,T)$.

Denote by $M(X, T)$ and $M^e(X, T)$ respectively the set of all invariant probability measures and
all ergodic invariant probability measures on $(X, T)$.
A dynamical system $(X,T)$ is called an \emph{E-system} if it is transitive and
there is an invariant probability measure $\mu$ with full support, i.e., $\supp(\mu)=X$;
an \emph{M-system} if it is transitive and the set of minimal points is dense in $X$.

Let $(X,T)$ and $(Y,S)$ be two dynamical systems.
If there is a continuous surjection $\pi\colon X\to Y$ which intertwines the actions (i.e., $\pi\circ T=S\circ \pi$),
then we say that $\pi$ is a \emph{factor map},
 $(Y,S)$ is a \emph{factor} of $(X,T)$ or $(X,T)$ is an \emph{extension} of $(Y,S)$.

\subsection{Dynamical properties via families}
The idea of using families to
describe dynamical properties goes back at least to Gottschalk and Hedlund~\cite{GH55}.
It was developed further by Furstenberg~\cite{F81}.
For a systematic study and recent results, see~\cite{A97}, \cite{G04}, \cite{HY04-2} and \cite{HY04}.

Let $(X,T)$ be a dynamical system.
For a point $x\in X$ and a subset $U$ of $X$, we define the
\emph{entering time set} of $x$ into $U$ by
\[N(x,U)=\{n\in\N\colon T^n(x)\in U\}.\]
Let $\F$ be a Furstenberg family.
A point $x\in X$ is called an \emph{$\F$-recurrent point},
if $N(x, U)\in\F$ for every open neighborhood $U$ of $x$.
It is well known that the following lemmas hold (see, e.g., \cite{A97,F81}).

\begin{lem} Let $(X, T)$ be a dynamical system and $x\in X$. Then
\begin{enumerate}
\item $x\in X$ is a minimal point if and only if it is an $\F_{\mathrm{s}}$-recurrent point.
\item $x \in X$ is a recurrent point if and only if it is an $\F_{\ip}$-recurrent point.
\end{enumerate}
\end{lem}

Recall that a dynamical system $(X,T)$ is
\emph{$\F$-transitive}
if for two non-empty open subsets $U$ and $V$ of $X$, $N(U,V) \in\F$;
\emph{$\F$-mixing} if $(X\times X, T\times T)$ is $\F$-transitive.

\begin{lem} Let $(X, T)$ be a dynamical system and $\F$ be a family. Then
\begin{enumerate}
\item $(X, T)$ is weakly mixing if and only if it is $\F_{\mathrm{t}}$-transitive.
\item $(X, T)$ is strongly mixing if and only if it is $\F_{\mathrm{cf}}$-transitive.
\item $(X,T)$ is $\F$-mixing if and only if it is $\F$-transitive and weakly mixing.
\end{enumerate}
\end{lem}

Let $\F$ be a family. A point $x\in X$ is called an \emph{$\F$-transitive point},
if $N(x, U)\in\F$ for every non-empty open subset $U$ of $X$ \cite{L11}.
The system $(X,T)$ is called \emph{$\F$-point transitive} if there exists some $\F$-transitive point.
It is clear that a dynamical system is transitive if and only if it is
$\F_{\mathrm{inf}}$-point transitive; minimal if and only if it is $\F_\mathrm{s}$-point transitive.
Denote $\Trans_{\F}(X,T)$ by the set of all $\F$-transitive points of $(X,T)$.

\begin{lem} [\cite{HY04-2,L11}] \label{thm:E-M-system}
Let $(X, T)$ be a dynamical system. Then
\begin{enumerate}
\item $(X,T)$ is an E-system if and only if $\Trans_{\F_{\pubd}}(X,T)=\Trans(X,T)\neq\emptyset$;
\item $(X,T)$ is an M-system if and only if $\Trans_{\F_{\ps}}(X,T)=\Trans(X,T)\neq\emptyset$.
\end{enumerate}
\end{lem}

\subsection{Hyperspace} \label{sec:hyperspace}
Let $(X,d)$ be a compact metric space.
Let $(2^X, d_H)$ be the \emph{hyperspace} of  $(X,d)$,
i.e., the collection of all non-empty closed subsets of $X$ endowed with the Hausdorff metric $d_H$.
We refer the reader to  \cite{A04} or \cite{IN99} more details on the hyperspace.
Let $(X,T)$ be a dynamical system. It induces naturally
a dynamical system $(2^X, \hat T)$ on the hyperspace,
where $\hat T(A)=T(A)$ for $A\in 2^X$.

Let $\Perf(X)$ denote the collection of all non-empty perfect subsets of $X$.
For every $n\in\N$
we define a subset $L_n(X)$ of $2^X$ as follows: $E\in L_n(X)$
if and only if there exist some $k\in\N$ and non-empty open subsets $\{U_i\}_{i=1}^k$ of $X$ such that
\begin{enumerate}
\item $\bigcup_{i=1}^k U_i\supset E$,
\item $\diam(U_i)<\frac{1}{n}$ and $\#(U_i\cap E)\geq 2$,  for $i=1,2,\ldots,k$.
\end{enumerate}
It is not hard to check that $L_n(X)$ is an open subset of $2^X$ and
$\Perf(X)=\bigcap_{n=1}^\infty L_n(X)$.
Then $\Perf(X)$ is a $G_\delta$ subset of $2^X$.
In addition, if $X$ is perfect, then $\Perf(X)$ is a dense $G_\delta$ subset of $2^X$
(see \cite{A04} or~\cite{BH08}).

Let $R$ be a relation of $n$-tuples on $X$, i.e., $R\subset X^n$.
A subset $K$ of $X$ with at least $n$ points is said to be
\emph{$R$-dependent}, or a \emph{dependent set} of $R$
if for any pairwise distinct $n$ elements $x_1,\ldots, x_n$ of $K$,
we have $(x_1,x_2,\ldots, x_n)\in R$.
Note that $K$ is $R$-dependent if and only if $K^n\setminus \Delta^{(n)}\subset R$,
where $\Delta^{(n)}=\{(x_1,x_2,\dotsc,x_n)\in X^n:\exists1\leq i<j\leq n,\text{ s.t. }x_i=x_j$\}.
An $X^n\setminus R$-dependent set is called an \emph{$R$-independent set} in \cite{My64}.
Suppose $\{R_n\}$ is a sequence of relations of $X$, i.e. $R_n\subset X^n$ for every $n\geq 1$.
A subset $K$ of $X$ is said to be \emph{$\{R_n\}$-dependent}, or a \emph{dependent set of $\{R_n\}$}
if $K$ is $R_n$-dependent for every $n$.
Let $\mathbb{D}(\{R_n\})$ denote the collection of all closed $\{R_n\}$-dependent sets.
We restate here a version of \cite[Proposition 4.3(h)]{A04} which we shall use.

\begin{prop} \label{prop:G-delta}
Let $(X,d)$ be a compact metric space without isolated points.
If $R$ is a $G_\delta$ subset of $X^n$,
then the collection of closed $R$-dependent sets is a $G_\delta$ subset of $2^X$.
\end{prop}

Recall that a map $\pi\colon X\to Y$ is called \emph{open}
if for every non-empty open subset $U$ of $X$, $\pi(U)$ is open in $Y$.
\begin{thm}[\mbox{\cite[Theorem 1.2]{A04}}] \label{lem:Ulam}
Let $\pi\colon X\to Y$
be an open continuous surjection between compact metric spaces.
If  $E\subset X$ is a dense $G_\delta$ set, then there exists a dense $G_\delta$ subset $Y_0$ of $Y$
such that $E\cap \pi^{-1}(y)$ is a dense $G_\delta$ subset of $\pi^{-1}(y) $ for each $y\in Y_0$
\end{thm}

\section{\texorpdfstring{$\F$}{F}-mixing sets}
Xiong and Yang \cite{XY92} characterized weak mixing by Xiong chaotic sets.
Huang, Shao and Ye \cite{HSY04} extended to Xiong-Yang¡¯s result to $\F$-mixing.
Inspired by Xiong chaotic sets, Blanchard and Huang~\cite{BH08} introduced
the localizaiton of weak mixing, weakly mixing sets,
and  characterized them by Xiong chaotic sets.
In this section,
we introduce the concept of $\F$-mixing sets and also characterize them by Xiong chaotic sets.
\begin{defn}
Let $(X,T)$ be a dynamical system and $\F$ be a Furstenberg family.
Suppose that $A$ is a closed subset of
$X$ with at least two points.
The set $A$ is said to be \emph{$\F$-mixing} if for any $k\in\N$,
any open subsets $U_1,U_2,\dots,U_k$, $V_1,V_2,\dotsc,V_k$ of $X$ intersecting $A$,
\[\bigcap_{i=1}^k N(U_i\cap A,V_i)\in \F.\]
\end{defn}

\begin{rem}
By the definition, it is not hard to see that an $\F$-mixing set must be perfect.
\end{rem}

We have the following characterization of $\F$-mixing sets.
The proof is in main part the same as the proof of  Theorem 5.2 and the Appendix of \cite{HSY04}
in the case of $\F$-mixing systems and the proof of Proposition 4.2 of \cite{BH08} in the case of weak mixing sets.
For the sake of completeness, we provide a proof here.

\begin{thm}\label{thm:F-Ming-set}
Let $(X,T)$ be a dynamical system and $\F$ be a Furstenberg family.
Suppose that $A$ is a closed subset of
$X$ with at least two points.
Then $A$ is an $\F$-mixing set if and only if for every $S\in\kappa\F$ (the dual family of $\F$)
there are Cantor subsets $C_1\subset C_2\subset \dotsc$ of $A$ such that
\begin{enumerate}
\item[(i)] $K=\bigcup_{n=1}^\infty C_n$ is dense in $A$;
\item[(ii)] for any $n\in\N$ and any continuous function $g\colon C_n\to A$
there exists a subsequence $\{q_i\}$ of $S$ such that
$\lim_{i\to\infty} T^{q_i}(x)=g(x)$ uniformly on $x\in C_n$;
\item[(iii)] for any subset $E$ of $K$ and
any continuous map $g \colon E\to A$ there exists a subsequence $\{q_i\}$ of $S$ such that
$\lim_{i\to\infty} T^{q_i}(x)=g(x)$ for every $x\in E$.
\end{enumerate}
\end{thm}

\begin{proof}
We first prove the sufficiency.
Fix $S\in \kappa\F$.
Choose $K=\bigcup_{n=1}^\infty C_n$ satisfying requirements with respect to $x$ and $S$.
Fix $k\in \mathbb{N}$ and non-empty open subsets  $U_1,\ldots, U_k$ and $V_1, \ldots, V_k$ of $X$
intersecting $A$.
Choose $x_i\in K\cap U_i$ and $y_i\in A\cap V_i$ for $1\leq i\leq k$.
Since $K$  has no isolated points, we can require $x_s\neq x_t$ for $1\leq s<t\leq k$.
Put $E=\{x_1,x_2,\ldots, x_k\}$ and take $g\colon E\to A$ with $g(x_i)=y_i, i=1,2,\ldots,k$.
Then there exists $\{n_j\}\subset S$ such that $T^{n_j}(x)\to g(x)$ for every $x\in E$.
As $E$ is finite and $y_i\in V_i$, there exists some $n_{j_0}\in S$ such that
$T^{n_{j_0}}(x_i)\in V_i$ for $i=1,\ldots,k$.
Then
\[\bigcap_{i=1}^k N(U_i\cap A,V_i)\cap S\neq\emptyset,\]
which implies that $\bigcap_{i=1}^k N(U_i\cap A,V_i)\in\kappa\kappa\F=\F$
since $S\in\kappa\F$ is arbitrary.

Now we prove the necessity.
Fix $S\in \kappa\F$.
Choose a sequence $\{O_j\}$ of non-empty open subsets of $X$ with $\diam(O_j)\to 0$ such that
$\{O_j\cap A\}$ is a countable topological base of $A$.
Let $Y=\{y_1,y_2,\ldots\}$ be a dense countable subset of $A$ and $Y_n=\{y_1,y_2,\ldots,y_n\}$.
Set $a_0=0$ and $U_0=X$.

\begin{claim}
There is an increasing sequence $\{a_n\}_{n=1}^\infty$ of positive integers and non-empty open subsets
$\{U_{n,1}, U_{n,2},\ldots, U_{n,a_n}\}_{n=1}^{\infty}$ of $X$ intersecting $A$  such that
\begin{enumerate}
\item[(1)] $2a_{n-1}\leq a_n\leq 2a_{n-1}+n$;
\item[(2)] $\diam (U_{n,i})<\frac{1}{n}$ for $i=1,2,\ldots,a_n$;
\item[(3)] the closure $\left\{\overline{U_{n,i}}\right\}_{i=1}^{a_n}$ are pairwise disjoint;
\item[(4)] $\overline{U_{n,2i-1}}\bigcup\overline{U_{n,2i}}\subset U_{n-1,i}$ for $i=1,2,\ldots, a_{n-1}$;
\item[(5)] $Y_n\subset B(\bigcup_{i=1}^{a_n}U_{n,i}, \frac{1}{n})$;
\item[(6)] for any $\beta\in\{1,2,\ldots, a_n\}^{a_n}$ there exists $m(\beta)\in S$ such that
\[T^{m(\beta)}(U_{n,i})\subset O_{\beta(i)} \text{ for }i=1,2,\ldots,a_n.\]
\end{enumerate}
\end{claim}
\begin{proof}[Proof of Claim 1]
Let $a_1=1$ and $U_{1,1}^{(0)}$ be a neighborhood of $y_1$ with $\diam(U_{1,1}^{(0)})<1$.
Then $U_{1,1}^{(0)}\cap A\neq\emptyset$ and $N(U_{1,1}^{(0)}\cap A,O_1)\in\F$.
Choose $m(1)\in S\cap N(U_{1,1}^{(0)}\cap A,O_1)$.
Then there exists a non-empty open subset $U_{1,1}$ of $U_{1,1}^{(0)}$
intersecting $A$ such that $T^{m(1)}(U_{1,1})\subset O_1$.

Assume that for $1\leq j\leq n-1$, we have
$\{a_j\}_{j=1}^{n-1}$ and $\{U_{j,1}, U_{j,2},\ldots, U_{j,a_j}\}_{j=1}^{n-1}$ satisfying conditions (1)--(6).
Since $A$ has no isolated points and the number of elements of $Y_n$ is $n$,
we can take $2a_{n-1}\leq a_n \leq 2a_{n-1}+n$ and non-empty open subsets
$U_{n,1}^{(0)},U_{n,2}^{(0)},\ldots, U_{n,a_n}^{(0)}$ of $X$ intersecting $A$ such that
\begin{enumerate}
\item[(7)] $\diam(U_{n,i}^{(0)})\leq \frac{1}{2n}$ for $i=1,2,\ldots, a_n$;
\item[(8)] the closure $\Bigl\{\overline{U_{n,i}^{(0)}}\Bigr\}_{i=1}^{a_n}$ are pairwise disjoint;
\item[(9)] $\overline{U_{n,2i-1}^{(0)}}\cup\overline{U_{n,2i}^{(0)}}\subset U_{n-1,i}$ for $i=1,2,\ldots, a_{n-1}$;
\item[(10)] $Y_n\subset B(\bigcup_{i=1}^{a_n}U_{n,i}^{(0)}, \frac{1}{2n})$.
\end{enumerate}

We arrange $\{1,2,\ldots,a_n\}^{a_n}$ as $\{\beta_i\}_{i=1}^{t_n}$, where $t_n=(a_n)^{a_n}$.
Since $A$ is $\F$-mixing,
\[F_1:=\bigcap_{i=1}^{a_n} N(U_{n,i}^{(0)}\cap A,O_{\beta_1(i)})\in\F.\]
Choose $m(\beta_1)\in S\cap F_1$ and then there exist non-empty open subsets
$U_{n,i}^{(1)}$ of $U_{n,i}^{(0)}$ intersecting $A$
such that
\[T^{m(\beta_1)}(U_{n,i}^{(1)})\subset O_{\beta_1(i)}, \text{ for }1\leq i\leq a_n.\]
Assume that for $1\leq j\leq t_n-1$ one has $m(\beta_1), m(\beta_2),\ldots,m(\beta_j)\in S$
and non-empty open subsets $U_{n,i}^{(0)}\supset U_{n,i}^{(1)}\supset U_{n,i}^{(2)}\supset\cdots\supset U_{n,i}^{(j)}$
intersecting $A$ such that
\[T^{m(\beta_h)}(U_{n,i}^{(h)})\subset O_{\beta_h(i)}, \text{ for  }1\leq i\leq a_n\text{ and }1\leq h\leq j.\]
Choose
\[m(\beta_{j+1}) \in S\cap \bigcap_{i=1}^{a_n} N(U_{n,i}^{(j)}\cap A,O_{\beta_{j+1}(i)}).\]
Then there exist non-empty open subsets $U_{n,i}^{(j+1)}$ of $U_{n,i}^{(j)}$ intersecting $A$ such that \[T^{m(\beta_{j+1})}(U_{n,i}^{(j+1)})\subset O_{\beta_{j+1}(i)}\text{ for }1\leq i\leq a_n.\]
By induction we have $\{m(\beta_j)\}_{j=1}^{t_n}$ and $\{U_{n,i}^{(j)}\}_{i=1}^{a_n}$, $j=0,1,\dotsc,t_n$.
Let $U_{n,i}=U_{n,i}^{(t_n)}$ for $1\leq i\leq a_n$.
Then $\{U_{n,i}\}_{j=1}^{a_n}$ satisfies conditions (1)--(6).
This ends the proof of the Claim 1.
\end{proof}

For $n\geq 1$, put $C_n=\bigcap_{j=n}^{\infty} \bigcup_{i=1}^{ (2^{j-n}a_n)}\overline{U_{j,i}}$,
where $\{a_n\}$ and $\{U_{j,i}\}$ are as in the Claim 1.
By (1)--(4), $C_n$ is a Cantor set.
Let $K=\bigcup_{n=1}^{\infty} C_n$. By (5), $K$ is dense in $A$.

\begin{claim}
For every $n\in\N$, $C_n$ satisfies the requirement of (ii) with respect to $S$.
\end{claim}
\begin{proof}[Proof of Claim 2]
Fix a continuous map $g\colon C_n\to A$ and $\varepsilon>0$.
Since $C_n$ is a compact set, there exists $m\geq n$ such that if
$x,y\in C_n$ and $d(x,y)\leq \frac{1}{m}$ then $d(g(x),g(y))< \varepsilon/2$.
For every $i\in\{1,2,\ldots, 2^{m-n}a_n\}$, we choose $x_i\in \overline{U_{m,i}}\cap C_n$.
Note that when $z_i\in \overline{U_{m,i}}$, one has $d(z_i,x_i)\leq 1/m$.
By the choice of $m$, one has $d(g(z_i),g(x_i))< \varepsilon/2$ when $z_i\in \overline{U_{m,i}}\cap C_n$.

By the construction of $\{O_j\}$, for every $i\in\{1,2,\ldots, 2^{m-n}a_n\}$
there exists some $n_i\in\N$ such that $g(x_i)\in O_{n_i}$ and $\diam(O_{n_i})<\varepsilon/2$.
Let $N=\max\{m,n_1,n_2,\ldots, \allowbreak n_{2^{m-n}a_n}\}$, $i(j)=[(j-1)/2^{N-m}]+1$ and
$\beta\in\{1,2,\ldots,a_N\}^{a_N}$ be any sequence with  $\beta(j)=n_{i(j)}$ for $j\in\{1,2,\ldots, 2^{N-n}a_n\}$.
By (6) there exists $k=m(\beta)\in S$ such that $T^k U_{N,j}\subset O_{\beta(j)}$
for every $j\in\{1,2,\ldots, 2^{N-n}a_n\}$.

For each $x\in C_n$, there exists  some $j\in\{1,2,\ldots, 2^{N-n}a_n\}$ such that $x\in \overline{U_{N,j}}$.
By the definition of $i(j)$, one has  $x\in \overline{U_{m,i(j)}}$, then $d(g(x),g(x_{i(j)}))<\varepsilon/2$.
Since $T^kx\in \overline{T^k(U_{N,j})}\subset \overline{O_{\beta(j)}}$, one has
$d(T^k(x), g(x_{i(j)}))\leq \diam ( \overline{O_{\beta(j)}} )<\varepsilon/2$. Thus,
\[d(g(x), T^kx)\leq d(g(x),g(x_{i(j)}))+d(T^kx, g(x_{i(j)}))<\varepsilon.\]
This ends the proof of Claim 2.
\end{proof}

\begin{claim}
$C$ satisfies the requirement of (iii) with respect to $S$.
\end{claim}
\begin{proof}[Proof of Claim 3]
Fix a subset $E$ of $K$ and a continuous map $g\colon E\to A$.
For any $n\geq 1$, let
\begin{align*}
  E_n=\bigl\{x\in E\colon  \exists i_x\in\{1,2,\ldots,n\}, a_x\in\{1,2,\ldots,a_n\}
   \text{ s.t. } x\in U_{n,a_x}\cap E\subset g^{-1}(O_{i_x})\bigr\}.
\end{align*}
It should be noticed that $E_n$ may be empty when $n$ is small.
Obviously $E_1\subset E_2\subset \cdots \subset E$ and $\bigcup_{n=1}^{\infty}E_n=E$.

When $E_n\neq\emptyset $, we rewrite the family of non-empty open sets
\[\left\{ U_{n,j}\colon \exists x\in E_n, i_x\in\{1,2,\ldots,n\}
\text{ s.t. } x\in U_{n,j}\cap E\subset g^{-1}(O_{i_x})\right\}\]
as
\[\left\{ U_{n,i_{n,1}}, U_{n,i_{n,2}}, \ldots  U_{n,i_{n,b_n}} \right\},\]
where $1\leq i_{n,1}<i_{n,2}<\dotsb<i_{n,b_n}\leq a_n$.

For each $n\in\mathbb N$ , let $\beta_n\in\{1,2,\ldots,a_n\}^{a_n}$ be any sequence with \[\beta_n(i_{n,j})=\max\left\{k\in\{1,2,\ldots,a_n\}\colon
U_{n,i_{n,j}}\cap E\subset g^{-1}(O_k)\right\}\text{ for } 1\leq j\leq b_n.\]
Let $q_n=m(\beta_n)\in S$ be as in the Claim 1.

For every $\varepsilon>0$, there exists some $N\in\mathbb N$
such that $\diam(O_n)<\varepsilon$ for every $n\geq N$.
Fix $x\in C$ and choose $t> N$ such that $O_t$ is a neighborhood of $g(x)$.
As $g$ is continuous, $g^{-1}(O_t)$ is an open neighborhood of $x$ in $E$.
Thus there are $n_t\geq t$ and $i_x\in\{1,2,\ldots,a_{n_t}\}$
such that $x\in U_{n_t, i_x}\cap E\subset g^{-1}(O_t)$ by (2).
By (4) and the definition of $K$, for each $j\in\mathbb N$ there is some $i_{x,j}\in \{1,2,\ldots, a_{n_t+j}\}$ such that
\[x\in U_{n_t+j,i_{x,j}}\cap E\subset U_{n_t,i_x}\cap E\subset g^{-1}(O_t).\]
Thus, $\beta_{n_t+j}(i_{x,j})\geq t$ for any $j\in\mathbb N$.
Moreover, by (6) and the definition of $\{q_n\}$ one has for $j\in\N$,
\[T^{q_{n_t+j}}(U_{n_t+j,i_{x,j}})= T^{m(\beta_{n_t+j})}(U_{n_t+j,i_{x,j}})\subset O_{\beta_{n_t+j}(i_{x,j})}\]
and
\[x\in U_{n_t+j,i_{x,j}}\cap E \subset g^{-1}(O_{\beta_{n_t+j}(i_{x,j})}).\]
So $T^{q_{n_t+j}}(x)\in O_{\beta_{n_t+j}(i_{x,j})}$ and $g(x)\in O_{\beta_{n_t+j}(i_{x,j})}$.
For each $j\in\mathbb N$, one has  $d(T^{q_{n_t+j}}(x), g(x))\leq \diam(O_{\beta_{n_t+j}(i_{x,j})})<\varepsilon$, since $\beta_{n_t+j}(i_{x,j})\geq t$. This implies that $\lim_{i\to\infty}T^{q_i}(x)=g(x)$.
\end{proof}
Hence the whole proof is finished.
\end{proof}

\begin{rem} \label{rem:wm-cantor}
By Theorem 3.1 in \cite{E70}, if $D$ is a closed subset of a Cantor set $C$
then any continuous map  $g\colon D\to X$ can be extended to a continuous map $\tilde{g}\colon C\to X$.
Then the condition (ii) in Theorem~\ref{thm:F-Ming-set} can be replaced by
\begin{enumerate}
  \item[(ii)$^\prime$] for any $n\in\N$, any closed subset $D$ of $C_n$ and any continuous function $g\colon D\to A$
there exists a subsequence $\{q_i\}$ of $S$ such that
$\lim_{i\to\infty} T^{q_i}(x)=g(x)$ uniformly on $x\in D$.
\end{enumerate}
\end{rem}

\begin{rem}
It should be noticed that Oprocha used the Kuratowski-Mycielski Theorem (see \cite[Theorem 5.10]{A04}) to
prove Theorem~\ref{thm:F-Ming-set} with Remark \ref{rem:wm-cantor} in case of weakly mixing sets
(see \cite[Corollary 48]{O11}).
\end{rem}

\section{Dynamical systems with positive topological entropy}
In this section, we study dynamical systems with positive topological  entropy.
We refer the reader to the textbook \cite{W82} for the theory of entropy.

Let $(X, T)$ be a dynamical system. For an open cover $\mathcal{U}$ of $X$, denote the topological entropy of
$\mathcal{U}$  by $\htop(T,\mathcal{U})$. The \emph{topological entropy} of $(X,T)$ is
\[ \htop(T) = \sup \htop(T , \mathcal{U}), \]
where the supremum is taken over all open covers $\mathcal{U}$.

Let $\mu$ be an invariant measure on $(X,T)$.
For a finite measurable partition $\alpha$ of $X$, denote the $\mu$-entropy of $\alpha$ by $h_\mu(T, \alpha)$.
The entropy of $\mu$ is
\[h_\mu(T ) = \sup h_\mu(T , \alpha), \]
where $\alpha$ ranges over all finite measurable partitions of $(X,\B_X)$.

The variational principle tells us that
\[\htop(T) = \sup_{\mu\in M^e (X,T )}h_\mu(T). \]

A measure-theoretical dynamical system is a Kolmogorov system
if and only if it has uniformly positive entropy (i.e., each finite non-trivial partition has positive entropy)
if and only if it has completely positive entropy (i.e., each of its non-trivial
factors has positive entropy) if and only if it is disjoint from every zero-entropy
system.
The notion of Kolmogorov system plays an important role in ergodic theory.
To get a topological analogy, Blanchard \cite{B92}
introduced the notions of complete positive entropy and uniformly positive entropy in topological dynamics.
He then naturally defined the notion of entropy pairs and used it to show that a uniformly positive entropy
system is disjoint from all minimal zero entropy systems \cite{B93}.
Later on, in \cite{BHMMR95} the authors were able to define entropy pairs for
an invariant measure and showed that for each invariant measure the set of entropy pairs for this
measure is contained in the set of entropy pairs. To obtain a better understanding of
the topological version of a Kolmogorov system, Huang and Ye \cite{HY06} introduced the notions of
entropy $n$-tuples ($n\geq 2$) both in topological and measure-theoretical settings.
See \cite{GY09} for a survey on the local entropy theory.

Let $K\subset X$ be a non-empty set and $\mathcal{U}$ be an open cover or a partition of $X$.
We say that $\mathcal{U}$ is \emph{admissible} with respect to $K$
if $\overline{U}\not\supset K$ for any $U\in \mathcal{U}$.

\begin{defn}[\cite{HY06}]
Let $(X,T)$ be a dynamical system, $\mu\in M(X,T)$ and $n\geq 2$.
An $n$-tuple $(x_1,x_2,\ldots,x_n)\in X^n\setminus\Delta_n(X)$ is called
\begin{enumerate}
\item a \emph{topological entropy $n$-tuple} if for any open cover $\mathcal{U}$ of $X$,
admissible with respect to $\{x_1,\ldots, x_n\}$, one has $\htop(T,\mathcal{U})>0$.
\item a \emph{$\mu$-entropy $n$-tuple} if for any measurable partition $\alpha$ of $X$,
admissible with respect to $\{x_1,\ldots, x_n\}$, one has $h_{\mu}(T,\alpha)>0$.
\end{enumerate}
\end{defn}

Denote by $E_n(X,T )$ the set of all topological entropy $n$-tuples and
by $E^\mu_n(X,T)$ the set of all entropy $n$-tuples for $\mu$.

Let $(X,T)$ be a dynamical system and $\mu\in M(X,T)$.
Let $\B_\mu$ be the completion of $\B_X$ under the measure $\mu$.
Then $(X,\B_\mu,\mu,T)$ is a Lebesgue system.
Let $\mathscr{P}_\mu$ be the Pinsker $\sigma$-algebra of $(X,\B_\mu,\mu,T)$ and
$\pi\colon(X,\B_\mu,\mu,T)\to (Y,\mathscr{D},\nu,S)$ be the measure-theoretical Pinsker factor of
$(X,\B_\mu,\mu,T)$.
The measure $\mu$ can be disintegrated over $(Y,\mathscr{D},\nu,S)$ as $\mu=\int_Y \mu_y d\nu(y)$,
where $\mu_y\in M(X,T)$ and $\mu_y(\pi^{-1}(y))=1$ for $\nu$-a.e. $y\in Y$ (see \cite[Theorem 5.8]{F81}).

For every $n\geq 1$, define
\[\lambda_n^\mu=\int_Y\mu_y^{(n)}d\nu(y),\]
where $\mu_y^{(n)}=\mu_y\times\cdots\times\mu_y$ ($n$-times).
Then $\lambda_n^\mu$ is an invariant Borel probability measure on $(X^n, T^{(n)})$.
There is a simple characterization of the set of $\mu$-entropy tuples.
\begin{prop}[\cite{G97,HY06}] \label{prop:supp-lambda}
Let $(X,T)$ be a dynamical system, $\mu\in M(X,T)$ and $n\geq 2$. Then
\[\supp(\lambda_n^\mu)\backslash\Delta_n(X)=E_n^\mu(X,T).\]
In addition, if $\mu$ is ergodic, then $\lambda_n^\mu$ is also ergodic
and $(\supp(\lambda_n^\mu),T^{(n)})$ is an E-system.
\end{prop}

In order to find where the entropy is concentrated,
the notion of entropy sets was introduced by Dou-Ye-Zhang \cite{DYZ06}
and Blanchard-Huang \cite{BH08}, independently.
In this paper, we follow the definition in \cite{BH08} which requires entropy sets to be closed.

\begin{defn}
Let $(X,T)$ be a dynamical system and $\mu\in M(X,T)$.
A closed subset $K$ of $X$ with at least two points is called
\begin{enumerate}
\item an \emph{entropy set} if for any open cover $\mathcal{U}$ of $X$, admissible with respect to $K$,
one has $\htop(T,\mathcal{U})>0$.
\item a \emph{$\mu$-entropy set} if for any measurable partition $\alpha$ of $X$, admissible with respect to $K$,
one has $h_{\mu}(T,\alpha)>0$.
\end{enumerate}
\end{defn}

Denote by $E_s(X,T )$ the collection of all entropy sets and
by $E^\mu_s(X,T)$ the collection of all $\mu$-entropy sets.
It follows immediately from definitions that a closed subset $K$ of $X$ with at
least two points is an entropy set (resp. $\mu$-entropy set) if and only if for every $n\geq 2$,
every $n$-distinct $n$ points $k_1,\dotsc,k_n\in K$ one has $(k_1,\dotsc,k_n)\in E_n(X,T)$
(resp. $(k_1,\dotsc,k_n)\in E_n^\mu(X,T)$.

Define
\[H(X,T)=\overline{E_s(X,T )}\text{ and }H^\mu(X,T)=\overline{E^\mu_s (X,T)}.\]
A point $x\in X$ is called an \emph{entropy point} if $\{x\}\in H(X,T)$;
a \emph{$\mu$-entropy point} if $\{x\}\in H^\mu(X,T)$.
The set of all entropy points and $\mu$-entropy points of $(X,T)$ is
denoted by $E_1(X,T)$ and $E_1^\mu(X,T)$, respectively.
\begin{thm}[\cite{BH08}] \label{thm:E1-mu}
Let $(X,T)$ be a dynamical system.
If $\mu\in M^e(X,T)$ with $h_\mu(X,T)>0$, then $E_1^\mu(X,T)=\supp(\mu)$.
\end{thm}

It is not hard to see that
\[H(X,T)=E_s(X,T)\cup\{\{x\}\colon x\in E_1(X,T)\}\]
and
\[H^\mu(X,T)=E^\mu_s(X,T)\cup\{\{x\}\colon x\in E_1^\mu(X,T)\}.\]
It is shown in \cite{BH08} that if an ergodic invariant measure $\mu\in M^e(X,T)$ has positive entropy,
then $(H^\mu(X,T),\hat T)$ is an E-system.
We could extend the result as follows.

\begin{thm} \label{thm:Hmu-Wm-E}
Let $(X,T)$ be a dynamical system.
If $\mu\in M^e(X,T)$ with $h_\mu(X,T)>0$,
then $(H^\mu(X,T), \hat T)$ is a weakly mixing E-system.
\end{thm}
\begin{proof}
We are left to show that $(H^\mu(X,T), \hat T)$ is weakly mixing.
By Proposition~\ref{prop:supp-lambda} and Theorem~\ref{thm:E1-mu},
$H^\mu(X,T)=\mathbb{D}(\supp(\lambda_n^{\mu}))$.
For every $n\geq 1$, let $R_n=\Trans(\supp(\lambda_n^{\mu}), T^{(n)})$.
Then $R_{n}$ is a dense $G_\delta$ subset of $R_\pi^{(n)}$,
and it is a $G_\delta$ subset of $X^n$ since $R_\pi^{(n)}$ is closed in $X^n$.
By Proposition~\ref{prop:G-delta}, $\mathbb{D}(R_n)$ is a $G_\delta$ subset of $2^X$.
For every $n\geq 1$, $R_n$ is dense in $\supp(\lambda_n^{\mu})$, then $\mathbb{D}(R_n)$ is dense in $H^\mu(X,T)$.
So $\mathbb{D}(R_n)$ is a dense $G_\delta$ subset of $H^{\mu}(X,T)$.
By \cite[Theorem 4.5]{BH08}, $\Perf(X)\cap H^\mu(X,T)$ is also a dense $G_\delta$ subset of $H^\mu(X,T)$.

Let $U_1,V_1,U_2,V_2$ be non-empty open subsets of $H^\mu(X,T)$.
There exist perfect sets $P_i'\in H^\mu(X,T)$ and $\varepsilon_1>0$
such that $B(P_i',3\varepsilon_1)\subset V_i$ for $i=1,2$.
Pick finite subsets $P_i''\subset P'_i$ such that
$B(P''_i,\varepsilon_1)\subset B(P'_i,3\varepsilon_1)$ for $i=1,2$.
Since $P'_1$ and $P'_2$ are perfect, we can require $P''_1\cap P''_2=\emptyset$.
Let $\varepsilon_2=\min\{d(x,y)\colon x\in P''_1,y\in P''_2\}>0$.
Choose a perfect set $P\in \mathbb{D}(R_n)$ and $\varepsilon_3>0$ such that
$B(P,\varepsilon_3)\subset B(P''_1\cup P''_2, \varepsilon_2/4)$.
Let $P'''_i=P\cap \{x\in X\colon d(x,P''_i)<\varepsilon_2/2\}$ for $i=1,2$.
By the choice of $\varepsilon_2$, we have $P'''_i$ are perfect and
$B(P'''_i,\varepsilon_3)\subset B(P''_i,\varepsilon_2/4)$ for $i=1,2$.
Then there exist $n\in\N$, $\varepsilon>0$ and
$P_i:=\{p_{i,1},p_{i,2},\dotsc,p_{i,n}\}\subset P'''_i$
such that $B(P_i, \varepsilon)\subset B(P'''_i, \varepsilon_3)\subset V_i$ for $i=1,2$.
Since $P_1\cup P_2\subset P\in \mathbb{D}(R_n)$,
we have
\[(p_{1,1},p_{1,2},\dotsc,p_{1,n},p_{2,1},p_{2,2},\dotsc,p_{2,n})\in R_{2n}.\]
Similarly, there exist $m\geq n$, $\delta>0$ and $Q_i:=\{q_{i,1},q_{i,2},\dotsc,q_{i,m}\}$
such that $B(Q_i, \delta)\subset U_i$ for $i=1,2$ and
\[(q_{1,1},q_{1,2},\dotsc,q_{1,m},q_{2,1},q_{2,2},\dotsc,q_{2,m})\in R_{2m}.\]
Let
\[U=\prod_{j=1}^m B(q_{1,j},\delta)\times \prod_{j=1}^m B(q_{2,j},\delta)\cap \supp(\lambda^\mu_{2m})\]
and
\[V=\prod_{j=1}^m B(p_{1,j},\delta)\times \prod_{j=1}^m B(p_{2,j},\delta)\cap \supp(\lambda^\mu_{2m}),\]
where $p_{i,n+1}=p_{i,n+2}=\dotsb=p_{i,m}=p_{i,n}$ for $i=1,2$.
Since $(\supp(\lambda^\mu_{2m}),T^{(2m)})$ is transitive, $N(U,V)\neq\emptyset$.
Then it is sufficient to show that
\[N(U,V)\subset N(U_1\times U_2,V_1\times V_2).\]
Pick a $k\in N(U,V)$.
There is $(z_1,z_2,\dotsc,z_{2m})\in U$ such that
$(T^k(z_1),T^k(z_2),\dotsc,\allowbreak T^k(z_{2n}))\in V$.
Let $Z_1=\{z_1,z_2,\dotsc,z_m\}$ and $Z_2=\{z_{m+1},z_{m+2},\dotsc,z_{2m}\}$.
Then $Z_i\in B(Q_i,\delta)\subset U_i$ and
$T^k(Z_i)\in B(P_i,\varepsilon)\subset V_i$ for $i=1,2$,
which imply that $k\in N(U_1\times U_2,V_1\times V_2)$.
\end{proof}

\begin{thm} \label{thm:positive-entorpy-mixing-set}
Let $(X,T)$ be a dynamical system and $\mu\in M^e(X,T)$ with $h_\mu (T)>0$.
Suppose that $\mu=\int \mu_yd\nu$ is the disintegration of $\mu$ over the Pinsker factor $(Y,\nu,S)$.
Then there is a Borel subset $Y_0$ of $Y$ with $\nu(Y_0)=1$ such that for every $y\in Y_0$,
$\supp(\mu_y)$ is an $\F_{\pubd}$-mixing set.
Moreover, the collection of $\F_{\pubd}$-mixing sets is residual in $H^\mu(X,T)$.
\end{thm}

\begin{proof}
For every $n\geq 1$, let $R_n=\Trans(\supp(\lambda_n^{\mu}), T^{(n)})$.
We do not know the structure of the collection of $\F_{\pubd}$-mixing sets,
but we can handle a subclass $Q_\mu$ of $\mu$-entropy sets.
A perfect $\mu$-entropy set $A$ is in $Q_\mu$ if and only if
for every $\varepsilon>0$ there exists a subset $E$ of $A$
such that $E$ is $\{R_n\}$-dependent and $d_H(E,A)<\varepsilon$.
By the proof of Theorem~\ref{thm:Hmu-Wm-E}, we know that
$\mathbb{D}(R_n)\cap \Perf(X)$ is a dense $G_\delta$ subset of $H^\mu(X,T)$.
Then $Q_\mu$ is residual in $H^\mu(X,T)$.
We are going to show that every element in $Q_\mu$ is an $\F_{\pubd}$-mixing set.

Let $A\in Q_\mu$.
Fix $n\in\N$ and open subsets $U_1,U_2,\dotsc,U_n$, $V_1,V_2,\dotsc,V_n$ of $X$ intersecting $A$.
Let
\[U=U_1\times U_2\times \dotsb \times U_n \text{ and } V=V_1\times V_2\times \dotsb \times V_n.\]
Since $A$ is perfect, there are $n$ distinct points $x_1,x_2,\dotsc,x_n\in A$ such that
$(x_1,x_2,\dotsc,\allowbreak x_n)\in U$.
By the definition of $Q_\mu$,
we can require $(x_1,x_2,\dotsc,x_n)$ to be a transitive point of $(\supp(\lambda_n^{\mu}),T^{(n)})$.
By Theorem \ref{thm:E-M-system}, $(x_1,x_2,\dotsc,x_n)$ is also an $\F_{\pubd}$-transitive point.
So $N((x_1,x_2,\dotsc,x_n),V)\in\F_{\pubd}$.
Then
\[N((x_1,x_2,\dotsc,x_n),V)\subset \bigcap_{i=1}^n N(U_i\cap A,V_i)\in \F_{\pubd},\]
which implies that $A$ is $\F_{\pubd}$-mixing.

Since $(\supp(\lambda^\mu_n), T^{(n)})$ is ergodic, by the Birkhoff Ergodic Theorem we have $\lambda^\mu_n(R_n)=1$.
Then there exists a Borel subset $Y_0$ of $Y$
with $\nu(Y_0)=1$ such that $\mu_y$ is non-atomic
and $\mu_y^{(n)}(R_n)=1$ for each $y\in Y_0$ and $n\geq 1$.
Then for every $y\in Y_0$, $\supp(\mu_y)$ is perfect since $\mu_y$ is non-atomic,
and $\mathbb{D}(\{R_n\})\cap 2^{\supp(\mu_y)}$ is dense in $2^{\supp(\mu_y)}$
since $R_n\cap (\supp(\mu_y))^n$ is dense in $(\supp(\mu_y))^n$
for every $n\geq 2$. Therefore, $\supp(\mu_y)\in Q_\mu$ for  every $y\in Y_0$.
\end{proof}

\begin{cor}
Let $(X,T)$ be a dynamical system.
If $\htop (T)>0$, then the collection of $\F_{\pubd}$-mixing sets is dense in $H(X,T)$.
\end{cor}
\begin{proof}
By Theorems 2.9 and 2.10 of \cite{BH08}, $\bigcup_{\mu\in M^e(X,T)} H^\mu(X,T)$ is dense in $H(X,T)$.
Then the result follows from Theorem~\ref{thm:positive-entorpy-mixing-set}.
\end{proof}

\begin{rem}
It is interesting to know that whether the collection of $\F_{\pubd}$-mixing sets is residual in $H(X,T)$.
It is shown in \cite[Theorem 4.4]{BH08} that the collection of weakly mixing sets is a $G_\delta$ subset of $2^X$.
But we do not know  whether the collection of $\F_{\pubd}$-mixing sets is a $G_\delta$ subset of $2^X$.
\end{rem}

\section{\texorpdfstring{$\F$}{F}-(point) mixing extensions}
In this section, we study factor maps between dynamical systems.
We show that if an intrinsic factor map is $\F$-point mixing,
then the collection of $\F$-mixing sets is residual in the fiber space.
We also apply this result to non-PI factors maps between minimal systems.

Let $\pi:(X,T)\to(Y,S)$ be a factor map between two dynamical systems.
For every $n\geq 2$, denote
\[R_\pi^{(n)}=\{(x_1,\ldots,x_n)\in X^n: \pi(x_1)=\cdots=\pi(x_n)\}.\]
Inspiring by the idea of entropy sets, we define the \emph{fiber space} of $\pi$ as
\[H_\pi=\bigcup_{y\in Y} 2^{\pi^{-1}(y)}= \mathbb D(\{R_\pi^{(n)}\}).\]
It is clear that $H_\pi$ is $\hat T$-invariant and it is closed by Lemma~\ref{lem:Hpi-closed}.
Then $(H_\pi,\hat{T})$ is a subsystem of $(2^X,\hat{T})$.
\begin{lem}\label{lem:Hpi-closed}
$H_\pi$ is a closed in $2^X$.
\end{lem}
\begin{proof}
Let $\{A_n\}$ be a sequence in $H_\pi$ with $A_n\to A$ as $n\to\infty$.
For every two points $a,b\in A$, there are $a_n, b_n\in A_n$ such that $a_n \to a, b_n\to b$ as $n\to\infty$.
Since $A_n$ is in $H_\pi$, we have $\pi(a_n)=\pi(b_n)$.
By continuity of $\pi$, $\pi (a)=\pi (b)$.
Then there exists $y\in Y$ such that $A\subset \pi^{-1}(y)$, which implies that $A\in H_\pi$.
\end{proof}

\begin{rem}
It is clear that for every $x\in X$, $\{x\}\in H_\pi$. Then $(X,T)$ can be regarded as a system of $(H_\pi,\hat{T})$.
If $\pi\colon X\to Y$ is a homeomorphism, then $H_\pi$ is homeomorphic to $X$.
This case is not interesting in our consideration.
We say that a factor map $\pi\colon (X,T)\to (Y,S)$ is \emph{intrinsic} if $\pi$ is not a homeomorphism.
\end{rem}

We say that the factor map $\pi\colon (X,T)\to (Y,S)$ is \emph{weak mixing} of order $n\in\mathbb{N}\setminus\{1\}$
if $(R_\pi^{(n)}, T^{(n)})$ is transitive;
\emph{weak mixing of all finite orders} if it is weakly mixing of order $m$ for every $m\geq 2$.
Note that there exists some factor maps which is weak mixing of order $2$ but not order $3$ (see \cite{G05} and \cite{OZ11}).

\begin{lem}\label{lem:wm-perf}
Let $\pi\colon (X,T)\to(Y,S)$ be an intrinsic factor map between two dynamical systems.
If $\pi$ is weakly mixing of all finite orders, then
$H_\pi\cap\Perf(X)$ is a dense $G_\delta$ subset of $H_\pi$.
\end{lem}
\begin{proof}
Recall that $\Perf(X)=\bigcap_{n=1}^\infty L_n$ and each $L_n$ is open.
It is sufficient to show that $L_n(X)\cap H_\pi$ is dense in $H_\pi$ for every $n\in\mathbb N$.
Now fix $n\in\mathbb N$, $\varepsilon>0$ and $E\in H_\pi$.
There exists a finite set $F=\{x_1,x_2,\ldots,x_k\}\in H_\pi$ such that $d_H(E,F)<\varepsilon$.
By the definition of $H_\pi$, one has
$(x_1,x_2,\ldots,x_k, x_1,x_2,\ldots,x_k)\in R_\pi^{(2k)}\backslash \Delta_{2k}(X)$.
Since $(R_\pi^{(2k)},T^{(2k)})$ is transitive,
there exists a transitive point $(y_1,y_2,\ldots,y_{2k})\in  R_\pi^{(2k)}$ such that
\[\max\{d(x_i,y_i),d(x_i,y_{k+i})\} <\min\{\tfrac{1}{2n}, \varepsilon\},\quad i=1,2,\ldots,k.\]
Then
\[d_H(\{y_1,y_2,\ldots,y_{2k}\}, E)\leq d_H(\{y_1,y_2,\ldots,y_{2k}\},F)+d_H(F,E)< 2\varepsilon.\]
Since $\pi$ is not a homeomorphism, for every $1\leq i<j\leq 2k $ there exists some point in $R_\pi^{(2k)}$ such that
its $i$'th coordinate is different from the $j$'th.
Then $y_1, y_2,\ldots, y_{2k}$ are pairwise distinct.
For $i=1,\dotsc,k$, $d(y_i,y_{k+i})\leq d(y_i,x_i)+d(x_i,y_{k+i})<\frac{1}{n}$,
then $\{y_1,y_2,\ldots,y_{2k}\}\in L_n(X)$. Therefore, $L_n(X)\cap H_\pi$ is dense in $H_\pi$.
\end{proof}

Let $\F$ be a Furstenberg family.
We say that the factor map $\pi\colon (X,T)\to (Y,S)$ is \emph{$\F$-mixing of all finite orders},
if $(R_\pi^{(n)}, T^{(n)})$ is $\F$-transitive for every $n\geq 2$.

\begin{thm}
Let $\pi: (X,T)\to (Y,S)$ be an intrinsic factor map and $\F$ be a Furstenberg family.
If $\pi$ is $\F$-mixing of all finite orders, then  $(H_\pi,\hat T)$ is $\F$-mixing.
\end{thm}

\begin{proof}
We use the same idea as in the proof of Theorem~\ref{thm:Hmu-Wm-E}.
Let $U_1,V_1,U_2,V_2$ be non-empty open subsets of $H_\pi$.
Then there exist $n\in\N$, $\varepsilon>0$ and
$P_i:=\{p_{i,1},p_{i,2},\dotsc,p_{i,n}\}$
such that $B(P_i, \varepsilon)\subset V_i$ for $i=1,2$, and
\[(p_{1,1},p_{1,2},\dotsc,p_{1,n},p_{2,1},p_{2,2},\dotsc,p_{2,n})\in R_\pi^{(2n)}.\]
Similarly, there exist $m\geq n$, $\delta>0$ and $Q_i:=\{q_{i,1},q_{i,2},\dotsc,q_{i,m}\}$
such that $B(Q_i, \delta)\subset U_i$ for $i=1,2$ and
\[(q_{1,1},q_{1,2},\dotsc,q_{1,m},q_{2,1},q_{2,2},\dotsc,q_{2,m})\in R_\pi^{(2m)}.\]
Let
\[U=\prod_{j=1}^m B(q_{1,j},\delta)\times \prod_{j=1}^m B(q_{2,j},\delta)\cap R_\pi^{(2m)}\]
and
\[V=\prod_{j=1}^m B(p_{1,j},\delta)\times \prod_{j=1}^m B(p_{2,j},\delta)\cap R_\pi^{(2m)},\]
where $p_{i,n+1}=p_{i,n+2}=\dotsb=p_{i,m}=p_{i,n}$ for $i=1,2$.
Since $(R_\pi^{(2m)},T^{(2m)})$ is $\F$-transitive, $N(U,V)\in \F$.
Then the result follows from
$N(U,V)\subset N(U_1\times U_2,V_1\times V_2)$.
\end{proof}

Let $\F$ be a Furstenberg family.
We say that the factor map $\pi\colon (X,T)\to (Y,S)$ is \emph{$\F$-point mixing of all finite orders},
if $\Trans_{\F}(R_\pi^{(n)}, T^{(n)})$ is a dense $G_\delta$ subset of $R_\pi^{(n)}$ for every $n\geq 2$.

\begin{thm}\label{thm:F-point-trans-extension}
Let $\pi: (X,T)\to (Y,S)$ be an intrinsic factor map and $\F$ be a Furstenberg family.
If $\pi$ is $\F$-point mixing of all finite orders, then the collection of $\F$-mixing sets is residual  in $H_\pi$.
Furthermore, if $\pi$ is open, then there is a dense $G_\delta$ subset $Y_0$ of $Y$,
such that $\pi^{-1}(y)$ is an $\F$-mixing set for every $y\in Y_0$.
\end{thm}
\begin{proof}
For every $n\geq 2$, let $R_{n}=\Trans_{\F}(R_\pi^{(n)}, T^{(n)})$.
Denote
\[Q_\pi=\{A\in H_\pi(X,T)\cap \Perf(X)\colon \forall \varepsilon>0,\ \exists E\in\mathbb{D}(R_n)\ \text{s.t. }
E\subset A, d_H(E,A)<\varepsilon\}.\]
By Lemma~\ref{lem:wm-perf},  $H_\pi\cap \Perf(X)$ is a dense $G_\delta$ subset of $H_\pi$.
For every $n\geq 2$, $R_{n}$ is a dense $G_\delta$ subset of $R_\pi^{(n)}$, and then
it is a $G_\delta$ subset of $X^n$ since $R_\pi^{(n)}$ is closed in $X^n$.
By Proposition~\ref{prop:G-delta}, $\mathbb{D}(\{R_{n}\})$ is a $G_\delta$ subset of $2^X$.
For every $H\in H_\pi$ and $\varepsilon>0$,
there exists a perfect set $H'\in H_\pi$ such that $d_H(H,H')<\varepsilon$.
Then there is a finite subset $H''=\{h_1,h_2,\dotsc,h_k\}\subset H'$ such that $k\geq 2$ and
$d_H(H',H'')<\varepsilon$.
Since $(h_1,h_2,\dotsc,h_k)\in R_\pi^{(k)}$, there exists $(r_1,r_2,\dotsc,r_k)\in R_k$
such that $\max\{d(h_i,r_i)\colon 1\leq i\leq k\}<\varepsilon$.
Then $\{r_1,r_2,\dotsc,r_k\}\in \mathbb{D}(\{R_{n}\})$ and $d_H(\{r_1,r_2,\dotsc,r_k\}, H)<3\varepsilon$,
which imply that $\mathbb{D}(\{R_{n}\})$ is dense in $H_\pi$.
Then $Q_{\pi}$ is residual in $H_\pi$, since it contains $\mathbb{D}(\{R_{n}\})\cap \Perf(X)$.
We are going to show that every element in $Q_\pi$ is an $\F$-mixing set.

Let $A\in Q_\pi$.
Fix $n\in\N$ and open subsets $U_1,U_2,\dotsc,U_n$, $V_1,V_2,\dotsc,V_n$ of $X$ intersecting $A$.
Let
\[U=U_1\times U_2\times \dotsb \times U_n \text{ and } V=V_1\times V_2\times \dotsb \times V_n.\]
Since $A$ is perfect, there are $n$ distinct points $x_1,x_2,\dotsc,x_n\in A$ such that
$(x_1,x_2,\dotsc,\allowbreak x_n)\in U$.
By the construction of $Q_\pi$,
we can choose $(x_1,x_2,\dotsc,x_n)$ to be an $\F$-transitive point in $(R^{(n)}_\pi,T^{(n)})$.
Then $N((x_1,x_2,\dotsc,x_n),V)\in\F$ and
\[N((x_1,x_2,\dotsc,x_n),V)\subset \bigcap_{i=1}^n N(U_i\cap A,V_i)\in \F,\]
which imply that $A$ is $\F$-mixing.

Now assume that $\pi$ is open.
Then for every $n\geq 2$, $\pi^{(n)}\colon R^{(n)}_\pi\to Y$, $(x_1,\dotsc,x_n)\mapsto \pi(x_1)$ is open.
By Theorem~\ref{lem:Ulam}, there exists a dense $G_\delta$ subset $Y_n$ of $Y$ such that
for every $y\in Y_n$, $(\pi^{-1}(y))^n\cap R_n$ is a dense $G_\delta$ subset of
$(\pi^{-1}(y))^n=(\pi^{(n)})^{-1}(y)$.
Let $Y_0=\bigcap_{n=2}^\infty Y_n$.
Fix $y\in Y_0$. Then $\pi^{-1}(y)$ is perfect.
If not, let $x$ be an  isolate point in $\pi^{-1}(y)$.
Then $(x,x)$ is an isolate point in $(\pi^{-1}(y))^2$.
But $(\pi^{-1}(y))^2\cap R_2$ is dense in $(x,x)$, so $(x,x)$ is a transitive pint in $(R_\pi^{(2)},T^{(2)})$,
which implies that $\pi$ is a homeomorphism.
This is a contradiction.
Using a similar argument for $A$, we can show that $\pi^{-1}(y)$ is $\F$-mixing.
\end{proof}

\begin{cor}
Let $(X,T)$ be a dynamical system.
\begin{enumerate}
\item If $(X,T)$ is a weakly mixing E-system, then
the collection of $\F_{\pubd}$-mixing sets is residual in $2^X$.
\item If $(X,T)$ is a weakly mixing M-system, then
the collection of $\F_{\ps}$-mixing sets is residual in $2^X$.
\end{enumerate}
\end{cor}
\begin{proof}
We only prove (1), since the proof of (2) is similar.
Assume that $(X,T)$ is a weakly mixing E-system.
It is clear that for every $n\geq 2$, $(X^n,T^{(n)})$ is also an E-system.
By Theorem \ref{thm:E-M-system}, $\Trans_{\F_{\pubd}}(X^n,T^{(n)})$ is a dense $G_\delta$ subset of $X^n$.
Then the result follows from
applying Theorem~\ref{thm:F-point-trans-extension} to the factor map from $(X,T)$ to the trivial system.
\end{proof}

We can also apply Theorem~\ref{thm:F-point-trans-extension} to factors maps between minimal systems.
We say that $\pi:(X,T)\to (Y,S)$ is a \emph{strictly PI extension} if there is an
ordinal $\eta$ (which is countable when $X$ is metrizable) and
a collection $\{\pi_\alpha^\beta:(X_\beta, T_\beta)\to (X_\alpha,T_\alpha)\mid \alpha\leq\beta\leq \iota\}$
of factor maps between minimal systems such that
\begin{enumerate}
\item $X_0=Y$, $T_0=S$, $X_\iota=X$ and $T_\iota=T$;
\item $\pi^\gamma_\alpha\circ \pi^\beta_\gamma=\pi^{\beta}_\alpha$ for $\alpha\leq\gamma\leq\beta\leq\iota$;
\item
if $\beta$ is a limit ordinal then $\pi^{\beta}_\alpha$ is the inverse limit of
$\{\pi^\gamma_\alpha\mid \alpha\leq\gamma<\beta\}$;
\item $\pi^{\alpha+1}_\alpha$ is either proximal or equicontinuous (isometric when $X$ is metrizable)
 for every $\alpha<\iota$.
\end{enumerate}
We say that $\pi:(X,T)\to (Y,S)$ is a \emph{PI extension} if there is a strictly PI extension
$\theta:(Z,R)\to (Y,S)$ and a proximal extension $\eta:(Z,R)\to (X,T)$ such that
$\theta= \pi\circ\eta$,
and a minimal system $(X,T)$ is \emph{PI} if
the extension from $(X,T)$ to the trivial system is PI.
We refer the reader to \cite{AGHSY10} or \cite{EGS} for the structure of minimal systems.

\begin{thm}
Let $\pi:(X,T)\to(Y,S)$ be a factor map between minimal systems. If $\pi$ is a non-PI extension,
then there is a dense subset $Y_0$ of $Y$ such that for every $y\in Y_0$
there exists some $\F_{\ps}$-mixing subset of $\pi^{-1}(y)$.
\end{thm}

\begin{proof}
By the structure theorem of minimal systems, there exist $\phi:(X_\infty, T_\infty)\to(X,T)$,
$\pi_\infty: (X_\infty, T_\infty)\to (Y_\infty, S_\infty)$ and $\eta: (Y_\infty, S_\infty)\to (Y,S)$
such that $\phi$ is a proximal extension, $\pi_\infty$ is a weakly mixing RIC extension,
and $\eta$ is a PI-extension. As $\pi$ is not PI, then $\pi_\infty$ is non-trivial.
Now consider the following commutative diagram.
\begin{equation*}
\xymatrix{
& {X}\ar[d]_\pi & {X_\infty}\ar[l]_{\phi}\ar[d]^{\pi_\infty}\\
& {Y} &{Y_\infty}\ar[l]^\eta}
\end{equation*}

By \cite[Theorem 2.7]{G05}, the extension $\pi_\infty$ is weakly mixing of all finite orders.
By \cite[Theorem A.2]{AGHSY10},
$(R_{\pi_\infty}^{(n)}, T_\infty^{(n)})$ is an $M$-system for every $n\geq 2$.
Then by Theorem \ref{thm:F-point-trans-extension}
there exists a dense $G_\delta$ subset $Y_\infty^0$ of $Y_\infty$ such that
$\pi^{-1}_\infty(y)$ is an $\F_{\ps}$-mixing set for every $y\in Y_\infty^0$.
Since $\pi$ is not PI, $\pi_\infty$ is not proximal.
Thus, there is a distal pair $(x_1,x_2)\in R_{\pi_\infty}^{(2)}$.
This implies that $\phi(x_1)\neq\phi(x_2)$, since $\phi$ is proximal.
Fix $y\in Y^0_\infty$. There exist $z_1,z_2\in \pi^{-1}_\infty(y)$
such that $(z_1,z_2)\in \Trans(R_{\pi_\infty}^{(2)})$. Then $\phi(z_1)\neq\phi(z_2)$,
since $(x_1,x_2)\in R_{\pi_\infty}^{(2)}$ and $\phi(x_1)\neq\phi(x_2)$.
So $\phi(\pi_\infty^{-1}(y))$ is not a singleton and then it is an $\F_{\ps}$-mixing subset of $X$.
Moreover $\phi(\pi_\infty^{-1}(y))\subset \pi^{-1}(\eta(y))$.
Finally, let $Y_0=\eta(Y_\infty^0)$, then $Y_0$ satisfies the requirement.
\end{proof}

\begin{cor}
If a minimal system is not PI, then it contains some $\F_{\ps}$-mixing set.
\end{cor}

\subsection*{Acknowledgments.}
The author would like to thank  Xiangdong Ye and Feng Tan for helpful suggestions and comments.
The authors express many thanks to the anonymous referees, whose remarks resulted in substantial
improvements of the paper.


\medskip
Received xxxx 20xx; revised xxxx 20xx.
\medskip

\end{document}